\documentclass[11pt,dvipdfmx]{amsart}
\usepackage[all]{xy}
\usepackage{bm,amsmath,amsthm,amssymb,mathtools,tikz,enumitem,comment}

\numberwithin{equation}{section}

\theoremstyle{plain}
\newtheorem{thm}{Theorem}[section]
\newtheorem*{thm*}{Theorem}
\newtheorem{lem}[thm]{Lemma}
\newtheorem*{lem*}{Lemma}
\newtheorem{prop}[thm]{Proposition}
\newtheorem*{prop*}{Proposition}
\newtheorem{cor}[thm]{Corollary}
\newtheorem*{cor*}{Corollary}

\theoremstyle{definition}

\newtheorem*{ex*}{Example}
\newtheorem{dfn}[thm]{Definition}
\newtheorem*{dfn*}{Definition}

\theoremstyle{remark}

\newtheorem*{rem*}{Remark}

\newcommand{\BC}{\mathbb{C}}
\newcommand{\BN}{\mathbb{N}}
\newcommand{\BQ}{\mathbb{Q}}
\newcommand{\BR}{\mathbb{R}}
\newcommand{\BZ}{\mathbb{Z}}

\newcommand{\Hom}{{\rm Hom}}

\newcommand{\Ker}{{\rm Ker}}

\newcommand{\GL}{{\rm GL}}

\usetikzlibrary{cd, decorations.pathmorphing}

\newcommand{\longtwoheadrightarrow}{\begin{tikzcd}[cramped,sep=scriptsize,ampersand replacement=\&]{}\arrow[r, two heads]\&{}\end{tikzcd}}

	\title{Classification of distinguished representations of GL(2)}

	\author{Yuki Matsumoto}
    \address{Department of Mathematics\\
    Graduate School of Science\\
    The University of Osaka\\
    Toyonaka, Osaka 560-0043\\
    Japan}
    \email{u005785d@ecs.osaka-u.ac.jp}

\begin{document}

\maketitle

\begin{abstract}
    Let $F$ be a non-archimedean local field with odd residual characteristic, and let $H$ be a maximal torus of $\GL_2(F)$. In this paper, we will classify the irreducible $H$-distinguished representations of $\GL_2(F)$ by using Casselman's criteria for $p$-adic symmetric spaces developed by Kato-Takano.
\end{abstract}

\section{Introduction}

Let $F$ be a non-archimedean local field with odd residual characteristic. Let $\bf G$ be a connected reductive group over $F$ with an $F$-involution $\sigma$, and $\bf H$ be the subgroup of all $\sigma$-fixed points of $\bf G$. For any $F$-subgroup $\bf R$ of $\bf G$, the group ${\bf R}(F)$ of $F$-points of $\bf R$ is denoted by $R$.

An admissible representation $(\pi,V)$ of $G$ is said to be $H$-distinguished if there exists a nonzero $H$-invariant linear form on $V$. In \cite{MR2428854}, Kato and Takano defined the $H$-matrix coefficient of an $H$-distinguished representation $(\pi,V)$ as follows: Let $v\in V$ and $\lambda \in \Hom_H(\pi,{\bf 1})$. The $H$-matrix coefficient $\varphi_{\lambda,v}$  is a function on symmetric space $G/H$ defined by
\begin{equation*}
	\varphi_{\lambda,v}(g)=\langle \lambda,\pi(g^{-1})v \rangle
\end{equation*}
for $g\in G$. Furthermore, the $H$-relatively cuspidal, square integrable and tempered representations were defined by Kato, Takano and Takeda (\cite{MR2428854, MR2566307, MR3889766}) as follows.

\begin{dfn} Let $(\pi,V)$ be an $H$-distinguished representation of $G$, and $Z_G$ be the center of $G$.
    \begin{enumerate}[leftmargin=*]
	    \item A representation $(\pi,V)$ is called $H$-relatively cuspidal if the $H$-matrix coefficients $\varphi_{\lambda,v}$ are compactly supported modulo $Z_GH$ for all $v\in V$ and $\lambda\in\Hom_H(\pi,{\bf 1})$.
    \end{enumerate}

    \noindent Assume $(\pi,V)$ admits a unitary central character.
    \begin{enumerate}[leftmargin=*, start=2]
	    \item A representation $(\pi,V)$ is called $H$-relatively square integrable if the $H$-matrix coefficients $\varphi_{\lambda,v}$ are square integrable modulo $Z_GH$ for all $v\in V$ and $\lambda\in\Hom_H(\pi,{\bf 1})$.

        \item A representation $(\pi,V)$ is called $H$-relatively tempered if the $H$-matrix coefficients $\varphi_{\lambda,v}$ are in $L^{2+\varepsilon}(G/Z_GH)$ for all $\varepsilon >0$, $v\in V$ and $\lambda\in\Hom_H(\pi,{\bf 1})$.
    \end{enumerate}
\end{dfn}

In this paper, we classify the irreducible $H$-distinguished representations of $G$ in the case ${\bf G} =\GL_2$ and $\sigma$ is a nontrivial inner involution of ${\bf G}$. Note that in this case, ${\bf H}$ is $G$-conjugate to
\begin{equation*}
	\left\{
	\begin{pmatrix}
		a & b \\ \tau b & a
	\end{pmatrix}
	\bigg |\ a^2-\tau b^2\neq 0\right\} \quad \text{for some}\ \tau\in F^{\times},
\end{equation*}
so that ${\bf H}$ is a maximal torus of ${\bf G}$. It is well known that the criterion proved by Waldspurger and Tunnell characterizes an infinite-dimensional irreducible admissible representation as being $H$-distinguished. We state the main results of this paper.

\begin{thm}\label{thm:classification_in_the_split_case} Let ${\bf G}$ be $\GL_2$, ${\bf H}$ be a maximal $F$-split torus of $G$, and $(\pi,V)$ be an irreducible $H$-distinguished representation of $G$.
\begin{enumerate}[leftmargin=*]
	\item A representation $(\pi,V)$ is $H$-relatively cuspidal if and only if it is either a supercuspidal or the untwisted Steinberg representation.

	\item A representation $(\pi,V)$ is $H$-relatively square integrable if and only if it is square integrable.

	\item A representation $(\pi,V)$ is $H$-relatively tempered if and only if it is tempered.
\end{enumerate}
\end{thm}

\begin{thm}\label{thm:classification_in_the_non-split_case} Let ${\bf G}$ be $\GL_2$, ${\bf H}$ be a maximal non $F$-split torus of $G$, and $(\pi,V)$ be an irreducible $H$-distinguished representation of $G$.
    \begin{enumerate}[leftmargin=*]
	    \item A representation $(\pi,V)$ is $H$-relatively cuspidal if and only if it is supercuspidal.

	    \item A representation $(\pi,V)$ is $H$-relatively square integrable if and only if it is square integrable.

	    \item A representation $(\pi,V)$ is $H$-relatively tempered if and only if it is tempered.
    \end{enumerate}
\end{thm}

The structure of this paper is as follows: In Section \ref{sec:structure_of_G}, we will recall some facts about $\sigma$-split parabolic subgroups and the real parts of characters. In section \ref{sec:H-dist_rep}, we will recall some classes of $H$-distinguished representations and theorems that characterize them. In section \ref{sec:induction_from_split_parabolic_subgp}, we will consider induced representations from split parabolic subgroups. In section \ref{sec:rep_GL_2}, we will recall the representation theory of $\GL_2(F)$. In Sections \ref{sec:main_result_in_the_split} and \ref{sec:main_result_in_the_non-split}, we will prove our main result.

\section*{acknowledgement}
This work is a revised version of the master's thesis of the author. The problem studied in this paper was suggested by the author's adviser Shuichiro Takeda. The author would like to thank him for his guidance and support throughout the work.

\section*{Notation and assumptions}\label{sec:notation}
Throughout, $F$ is a non-archimedean local field with odd residual characteristic. We let $\mathcal{O}_F,\varpi_F,q_F$ and $|\cdot|_F$ be the ring of integers, a uniformizer, the number of elements in the residue field and the norm of $F$, respectively.

Let $\bf G$ be a connected reductive group over $F$ with an $F$-involution $\sigma$, and $\bf H$ be the subgroup of all $\sigma$-fixed points of $\bf G$. For an $F$-subgroup $\bf R$ of $\bf G$, the group ${\bf R}(F)$ of $F$-points of $\bf R$ is denoted by $R$. If $P$ is a parabolic subgroup of $G$, we denote the modulus character of $P$ by $\delta_P$. Also, if $(\pi,V)$ is an admissible representation of $G$, then we denote by $(\pi_P,V_P)$ the normalized Jacquet module of $(\pi,V)$ along $P$.

\section{$\sigma$-split parabolic subgroups and real parts of characters}\label{sec:structure_of_G}
In this section, we will recall some facts about $\sigma$-split parabolic subgroups and the real parts of characters. For details, see \cite{MR2428854} and \cite{MR3889766}.

\subsection{Root systems}\label{subsec:root_systems}
Let $\bf P$ be a parabolic subgroup of $\bf G$ with Levi decomposition $\bf P=MU$, where $\bf M$ is the Levi factor and $\bf U$ is the unipotent radical. Then $\bf P$ is said to be $\sigma$-split if ${\bf P}\cap \sigma({\bf P})=\bf M$.

Let $\bf S$ be an $F$-split torus of $\bf G$.\ Then $\bf S$ is said to be $\sigma$-split if $\sigma(s)=s^{-1}$ for all $s\in\bf S$. Let ${\bf S}_0$ be a maximal $\sigma$-split torus. Furthermore, let $\bf A$ be a maximal $F$-split torus containing ${\bf S}_0$. Both $\bf A$ and ${\bf S}_0$ are $\sigma$-stable, and hence $\sigma$ naturally acts on the groups of rational characters
\begin{equation*}
	X^{\ast}({\bf A})\coloneqq \Hom({\bf A},\mathbb{G}_m),\quad X^{\ast}({\bf S}_0)\coloneqq \Hom({\bf S}_0,\mathbb{G}_m)
\end{equation*}
of the respective tori.

Define
\begin{equation*}
	\mathfrak{a}^{\ast}\coloneqq X^{\ast}({\bf A})\otimes \BR,\quad \mathfrak{s}^{\ast}\coloneqq X^{\ast}({\bf S}_0)\otimes \BR.
\end{equation*}
In this paper, we often write $\bar{\nu}=p(\nu)$ for $\nu\in \mathfrak{a}^{\ast}$ with respect to the projecton
\begin{equation*}
	p:\mathfrak{a}^{\ast}\to \mathfrak{s}^{\ast}.
\end{equation*}

Let $\Phi=\Phi(\bf G,A)\subseteq X^{\ast}({\bf A})$ be the set of roots of $\bf (G,A)$. We choose a set $\Delta$ of simple roots such that the corresponding ordering satisfies the property
\begin{equation*}
	\alpha >0\quad {\rm and}\quad \sigma(\alpha)\neq \alpha\quad \Longrightarrow\quad \sigma(\alpha)<0.
\end{equation*}
We define
\begin{equation*}
	\Delta^{\sigma}\coloneqq \{\alpha\in\Delta\ |\ \sigma(\alpha)=\alpha\}\quad {\rm and}\quad \overline{\Delta}\coloneqq p(\Delta)\smallsetminus \{0\}.
\end{equation*}

\subsection{$\sigma$-split parabolic subsets}
Recall that each standard parabolic subgroup corresponds to a subset of $\Delta$. Therefore, we call a subset of $\Delta$  which corresponds to a $\sigma$-split parabolic subgroup a $\sigma$-split parabolic subset. Moreover, it is known to correspond, in turn, to a subset of $\overline{\Delta}$. Through this correspondence, $\Delta$ corresponds to a maximal $\sigma$-split parabolic subgroup $\bf G$, while $\Delta^{\sigma} $ corresponds to a minimal $\sigma$-split parabolic subgroup.

We call the $\sigma$-split parabolic subgroup that corresponds to a $\sigma$-split parabolic subset $I$ a $\Delta$-standard $\sigma$-split parabolic subgroup. When $\Delta$ is understood from the context, we simply call it the $\Delta$-standard $\sigma$-split parabolic subgroup.

Recall that any parabolic subgroup can be made standard by choosing an appropriate simple roots. The same is true for $\sigma$-split parabolic subgroups. Namely

\begin{lem}Any $\sigma$-split parabolic subgroup can be made standard by choosing an appropriate simple roots.\end{lem}
\begin{proof}See {\cite[p.11]{MR2428854}}.\end{proof}

Once $\Delta$ is fixed, any parabolic subgroup is conjugate to a standard one. However,  this does not generally hold for $\sigma$-split parabolic subgroups. Kato and Takano have shown the following.

\begin{lem}\label{lem:conjugate_of_sigma-split}Let ${\bf M}_0$ be the Levi factor of the minimum $\sigma$-split parabolic subgroup ${\bf P}_{\Delta^{\sigma}}$.

	\noindent(1)\ Any $\sigma$-split parabolic subgroup is of the form $\gamma^{-1}{\bf P}_I\gamma$ for some $\sigma$-split parabolic subset $I\subseteq \Delta$ and $\gamma\in({\bf M}_0{\bf H})(F)$.

	\noindent(2)\ If $({\bf M}_0{\bf H})(F)=M_0H$, then every $\sigma$-split parabolic subgroup is $H$-conjugate to a standard $\sigma$-split one.
\end{lem}
\begin{proof}
	See \cite[Lemma\ 2.5.\ p.11]{MR2428854}.
\end{proof}

\subsection{Positive cones}\label{sec:positive_cones}
For a subset $I\subseteq \Delta$, let ${\bf P}$ be the parabolic subgroup ${\bf P}_I$ with $\bf M_P$ as the Levi factor and $\bf U_P$ as the unipotent radical as usual. Let $\bf A_P$ be the split component of center of $\bf M_P$. Then we have
\begin{equation*}
	{\bf A_P}=\bigg (\bigcap_{\alpha\in I}\Ker \alpha \bigg)^{\circ}\quad {\rm and}\quad {\bf M_P}=Z_{\bf G}(\bf A_P).
\end{equation*}
Moreover, if $I$ is a $\sigma$-parabolic subset, let ${\bf S}_0$ be a maximal $\sigma$-split torus contained in $\bf P$. We set
\begin{equation}\label{eq:S_P}
	{\bf S_P}\coloneqq {\bf S}_0\cap {\bf A_P}.
\end{equation}

For the group of rational characters $X^{\ast}({\bf S_P})\coloneqq \Hom({\bf S_P},\mathbb{G}_m)$ of $\bf S_P$, define
\begin{equation*}
	\mathfrak{s}_{\bf P}^{\ast}\coloneqq X^{\ast}({\bf S_P})\otimes \BR.
\end{equation*}

Let $\rho_P:\mathfrak{s}^{\ast}\to\mathfrak{s}_{\bf P}^{\ast}$ be the projection on $\mathfrak{s}_{\bf P}^{\ast}$, and define
\begin{align*}
	+\mathfrak{s}_{\bf P}^{\ast}&\coloneqq \{\nu\in \mathfrak{s}_{\bf P}^{\ast}\ |\ \nu=\sum_{\alpha\in\Delta\smallsetminus I}c_{\alpha}\rho_{\bf P}(\alpha),\ c_{\alpha}>0\},\\
	\overline{+\mathfrak{s}_{\bf P}^{\ast}}&\coloneqq \{\nu\in \mathfrak{s}_{\bf P}^{\ast}\ |\ \nu=\sum_{\alpha\in\Delta\smallsetminus I}c_{\alpha}\rho_{\bf P}(\alpha),\ c_{\alpha}\ge0\}.
\end{align*}

\subsection{Real parts of characters}
In this subsection, we will recall the notion of the real part of a quasi-character of a reductive group. Throughout this subsection, we let ${\bf G}$ be an arbitrary connected reductive group over $F$.

We set
\begin{equation*}
	\mathfrak{a}_{\bf G}^{\ast}\coloneqq X^{\ast}({\bf G})\otimes \BR\quad{\rm and}\quad \mathfrak{a}_{{\bf G},\BC}^{\ast}\coloneqq X^{\ast}({\bf G})\otimes \BC.
\end{equation*}
Furthermore we set
\begin{equation*}
	\widehat{G}\coloneqq \Hom_{cont}(G,\BC^{\times})\quad{\rm and}\quad \widehat{G}^{un}\coloneqq \Hom_{cont}(G/G^1,\BC^{\times}),
\end{equation*}
where $G^1=\cap_{\chi\in X^{\ast}(\bf G)}\Ker |\chi|_F$. It is well known to that there is a surjective group homomorphism
\begin{equation*}
	\mathfrak{a}_{{\bf G},\BC}^{\ast} \longtwoheadrightarrow \widehat{G}^{un},\quad \sum_i\chi_i\otimes s_i\mapsto (g\mapsto \prod_i |\chi_i(g)|^{s_i}),
\end{equation*}
for $\chi_i\in X^{\ast}({\bf G})$ and $s_i\in \BC$. Moreover the kernel of this homomorphism is the lattice $L$ of the form $L=(2\pi i/q_F)L_0$ for some lattice $L_0\subseteq X^{\ast}({\bf G})\otimes \BQ$, so we have the isomorphism
\begin{equation*}
	\mathfrak{a}_{{\bf G},\BC}^{\ast}/L \stackrel{\sim}{\longrightarrow} \widehat{G}^{un}.
\end{equation*}
With the choice of $L$, we have a well-defined map
\begin{equation*}
	{\rm Re}:\mathfrak{a}_{{\bf G},\BC}^{\ast}/L \longrightarrow \mathfrak{a}_{\bf G}^{\ast},\quad \sum_i\chi_i\otimes s_i\mapsto \sum_i\chi_i\otimes {\rm Re}(s_i),
\end{equation*}
for $\chi_i\in X^{\ast}({\bf G})$ and $s_i\in \BC$. For every $\chi\in\widehat{G}$, we have $|\chi|\in \widehat{G}^{un}$. Hence we define ${\rm Re}(\chi)$ to be the image of the composite
\begin{equation*}
	\widehat{G} \stackrel{|\cdot|}{\longrightarrow} \widehat{G}^{un} \stackrel{\sim}{\longrightarrow} \mathfrak{a}_{{\bf G},\BC}^{\ast}/L \stackrel{{\rm Re}}{\longrightarrow} \mathfrak{a}_{\bf G}^{\ast},
\end{equation*}
and we call it the real part of $\chi$.

\section{$H$-distinguished representations}\label{sec:H-dist_rep}
In this section, we will recall some classes of $H$-distinguished representations and theorems that characterize them. For details, see \cite{MR2428854}, \cite{MR2566307} and \cite{MR3889766}.

\subsection{$H$-distinguished representations}
An admissible representation $(\pi,V)$ of $G$ is said to be $H$-distinguished if $\Hom_{H}(\pi,{\bf 1})\neq0$.

\begin{dfn}Let $(\pi,V)$ be an $H$-distinguished representation of $G$. For nonzero $\lambda\in\Hom_{H}(\pi,{\bf 1})$ and $v\in V$, we define $\varphi_{\lambda,v}: G/H\to \BC$ by
	\begin{equation*}
		\varphi_{\lambda,v}(g)=\langle \lambda,\pi(g^{-1})v \rangle
	\end{equation*}
for $g\in G$. We call $\varphi_{\lambda,v}$ an $H$-matrix coefficient or $(H,\lambda)$-matrix coefficient.
\end{dfn}

\subsection{The map $r_P$}
In this subsection, we will recall the map $r_P$ defined by Kato-Takano in \cite{MR2428854}. For details, see Section 5 of \cite{MR2428854}.

We fix $({\bf S}_0,{\bf A},\Delta)$ in subsection \ref{subsec:root_systems}. Let $I\subsetneq \Delta$ be a $\sigma$-split parabolic subset, and let $(\pi,V)$ be an $H$-distinguished representation of $G$. Then there exists a linear map
\begin{equation*}
	r_{P_I}:\Hom_H(\pi,{\bf 1})\longrightarrow\Hom_{M_I^{\sigma}}(\pi_{P_I},{\bf 1})
\end{equation*}
on the space of invariant linear forms such that if $v\in V$ is a ``canonical lift'' (see \cite[5\ p.19]{MR2428854}) of $\bar{v}\in V_{P_I}$, then
\begin{equation*}
	\langle r_{P_I}(\lambda),\bar{v} \rangle = \langle \lambda,v \rangle
\end{equation*}
for all $\lambda \in \Hom_H(\pi,{\bf 1})$.

The map $r_P$ satisfies the following formula.

\begin{prop}\label{prop:r_P} Let $\lambda\in\Hom_H(\pi,{\bf 1})$. For any $v\in V$, there exists a positive number $\varepsilon\le1$ such that
    \begin{equation*}
		\langle \lambda,\pi(s)v \rangle =\delta_{P_I}^{1/2}(s)\langle r_{P_I}(\lambda),\pi_{P_I}(s)j_{P_I}(v) \rangle
	\end{equation*}
for all $s\in S_{P_I}^{-}(\varepsilon)$, where $j_{P_I}$ is the canonical map $V\to V_{P_I}$ and $S_{P_I}^{-}(\varepsilon)$ is defined by
    \begin{equation*}
		S_{P_I}^{-}(\varepsilon)\coloneqq  \{s \in  S_{P_I}\ |\ | s^{\alpha}|_F\le \varepsilon,\ \forall \alpha\in \Delta\smallsetminus I \}.
	\end{equation*}
\end{prop}
\begin{proof}
	See \cite[Proposition\ 5.5.\ p.22]{MR2428854}.
\end{proof}

\subsection{Relative Exponents}
In this subsection, we will recall the notion of ``relative exponents'' introduced by Kato-Takano in \cite{MR2566307}.

Let $(\pi,V)$ be a finitely generated admissible representation of $G$, let $Z$ be a closed subgroup of center $Z_G$ of $G$, and let $\widehat{Z}$ be the quasi-characters of $Z$. For each $\chi\in \widehat{Z}$, we define
\begin{equation*}
	V_{\chi}\coloneqq \{v\in V\ |\ \text{There exists}\ d\in \BN\  \text{such that}\ (\pi(z)-\chi(z))^dv=0\ \text{for all}\ z\in Z \}.
\end{equation*}
Namely $V_{\chi}$ is the generalized $\chi$-eigenspace.

Moreover, let $(\pi,V)$ be $H$-distinguished. Let $\bf P=MU$ be a $\sigma$-split parabolic subgroup and let ${\bf S}_{\bf P}$ be the set defined by (\ref{eq:S_P}). For each $\lambda\in\Hom_H(\pi,{\bf 1})$, we define
\begin{align*}
	\mathcal{E}xp_{S_P}(\pi_P,\lambda)&=\{\chi\in \widehat{S_P}\ |\ r_P(\lambda)\neq 0\ {\rm on}\  V_{P,\chi}\},\\
	{\rm Exp}_{S_P}(\pi_P,\lambda)&=\{{\rm Re}(\chi)\in\mathfrak{s}_{{\bf P}}^{\ast}\ |\ \chi\in \mathcal{E}xp_{S_P}(\pi_P,\lambda)\}.
\end{align*}
We call each element in $\mathcal{E}xp_{S_P}(\pi_P,\lambda)$ an $(H,\lambda)$-relative central exponent of $\pi$ along $P$, and each element in ${\rm Exp}_{S_P}(\pi_P,\lambda)$ an $(H,\lambda)$-relative exponent of $\pi$ along $P$.

\subsection{Classes of $H$-distinguished representations}
In this subsection, we will recall $H$-relatively cuspidal representations, $H$-relatively square integrable representations and $H$-relatively tempered representations. Furtheremore, we will recall the theorems that characterize them.

\begin{dfn} Let $(\pi,V)$ be an $H$-distinguished representation of $G$.
    \begin{enumerate}[label=\arabic*., leftmargin=*]
		\item
        \begin{enumerate}
            \item For each nonzero $\lambda\in\Hom_H(\pi,{\bf 1})$, a representation $(\pi,V)$ is said to be $(H,\lambda)$-relatively cuspidal if the $(H,\lambda)$-matrix coefficients $\varphi_{\lambda,v}$ are compactly supported modulo $Z_GH$ for all $v\in V$.

            \item A representation $(\pi,V)$ is said to be $H$-relatively cuspidal if it is $(H,\lambda)$-relatively cuspidal for all $\lambda\in\Hom_H(\pi,{\bf 1})$.
        \end{enumerate}
    \end{enumerate}

    \noindent Assume $(\pi,V)$ admits a unitary central character.
    \begin{enumerate}[label=\arabic*., leftmargin=*, start=2]
        \item
        \begin{enumerate}
        \item For each nonzero $\lambda\in\Hom_H(\pi,{\bf 1})$, a representation $(\pi,V)$ is said to be $(H,\lambda)$-relatively square integrable if the $(H,\lambda)$-matrix coefficients $\varphi_{\lambda,v}$ are square integrable modulo $Z_GH$ for all $v\in V$.

        \item A representation $(\pi,V)$ is said to be $H$-relatively square integrable if it is $(H,\lambda)$-relatively square integrable for all $\lambda\in\Hom_H(\pi,{\bf 1})$.
        \end{enumerate}

        \item
        \begin{enumerate}
            \item For each nonzero $\lambda\in\Hom_H(\pi,{\bf 1})$, a representation $(\pi,V)$ is said to be $(H,\lambda)$-relatively tempered if the $(H,\lambda)$-matrix coefficients $\varphi_{\lambda,v}$ are in $L^{2+\varepsilon}(G/Z_GH)$ for all $\varepsilon >0$ and $v\in V$.

            \item A representation $(\pi,V)$ is said to be $H$-relatively tempered if it is $(H,\lambda)$-relatively tempered for all $\lambda\in\Hom_H(\pi,{\bf 1})$.
        \end{enumerate}
    \end{enumerate}
\end{dfn}

The following theorem characterizes $H$-relatively cuspidal representations.
\begin{thm}\label{thm:criterion_of_relatively_cusp} Let $(\pi,V)$ be an $H$-distinguished representation of $G$ and $\lambda\in\Hom_H(\pi,{\bf 1})$. Then it is $(H,\lambda)$-relatively cuspidal if and only if $r_P(\lambda)=0$ for all proper $\sigma$-split parabolic subgroups $P$ of $G$. Furthermore, if all $\sigma$-split parabolic subgroups are $H$-conjugate to a standard one, then the condition has to be checked only for the standard $\sigma$-split parabolic subgroup.
\end{thm}
\begin{proof}
	See \cite[Theorem\ 6.2.\ p.26]{MR2428854}.
\end{proof}

As a consequence, we have the following proposition.
\begin{prop}\label{prop:relatively_cuspidal} Supercuspidal $H$-distinguished representations are $H$-relatively cuspidal.
\end{prop}
\begin{proof}
	See \cite[Proposition\ 8.1.\ p.32]{MR2428854}.
\end{proof}

Next, we will recall ``Casselman criteria'' for $H$-relative square square integrability and $H$-relative temperedness.

\begin{thm} Let $(\pi,V)$ be a finitely generated $H$-distinguished representation of $G$ which admits a unitary central character. Fix $\lambda\in\Hom_H(\pi,{\bf 1})$.
\begin{enumerate}[leftmargin=*]
		\item A representation $(\pi,V)$ is $(H,\lambda)$-relatively square integrable if and only if for all $\sigma$-split parabolic subgroups $P$, we have
		\begin{equation*}
			\nu\in+\mathfrak{s}_{\bf P}^{\ast}
		\end{equation*}
		for all $\nu\in{\rm Exp}_{S_P}(\pi_P,\lambda)$.

		\item A representation $(\pi,V)$ is $(H,\lambda)$-relatively tempered if and only if for all $\sigma$-split parabolic subgroups $P$, we have
		\begin{equation*}
			\nu\in\overline{+\mathfrak{s}_{\bf P}^{\ast}}
		\end{equation*}
		for all $\nu\in{\rm Exp}_{S_P}(\pi_P,\lambda)$.
	\end{enumerate}
Furthermore, if all $\sigma$-split parabolic subgroups are $H$-conjugate to a standard one, then the conditions have to be checked only for the standard $\sigma$-split parabolic subgroup.
\end{thm}
\begin{proof}
	See \cite[Theorem\ 4.7.\ p.1445]{MR2566307} and \cite[Theorem\ 3.12.\ p.282]{MR3889766}.
\end{proof}

As a consequence, we have the following corollaries.

\begin{cor}\label{cor:relatively_square} Let $(\pi,V)$ be a finitely generated $H$-distinguished representation which admits a unitary central character. If $(\pi,V)$ is square integrable, then it is $H$-relatively square integrable.
\end{cor}
\begin{proof} See \cite[Proposition\ 4.10.\ p.1441]{MR2566307}.
\end{proof}

\begin{cor}\label{cor:relatively_tempered} Let $(\pi,V)$ be a finitely generated $H$-distinguished representation which admits a unitary central character. If $(\pi,V)$ is tempered, then it is $H$-relatively tempered.
\end{cor}
\begin{proof} See \cite[Corollary\ 3.13.\ p.282]{MR3889766}.
\end{proof}

By Proposition \ref{prop:relatively_cuspidal}, Corollary \ref{cor:relatively_square}, and Corollary \ref{cor:relatively_tempered}, the following proposition follows immediately.
\begin{prop} Let $(\pi,V)$ be an $H$-distinguished representation of $G$.

    \noindent (1) Every superuspidal $H$-distinguished representation is $H$-relatively cuspidal.

    \noindent (2) Every irreducible square integrable $H$-distinguished representation is $H$-relatively square integrable.

    \noindent (3) Every irreducible tempered $H$-distinguished representation is $H$-relatively tempered.
\end{prop}

\section{Induced representations from split parabolic subgroups}\label{sec:induction_from_split_parabolic_subgp}
In this section, we will consider induced representations from split parabolic subgroups. For details, see \cite{MR4342358}.

We fix a $\sigma$-split parabolic subgroup $P=MU$. It is well known that $PH$ is an open in $G$ (see \cite[1.3.\ Theorem]{Vust1974OprationDG}). Let $\varXi_0$ be a (finite) set of representatives of all open $(P,H)$-double cosets in $G$ with $e\in\varXi_0$. We set $\mathcal{O}=\bigcup_{\xi\in\varXi_0}P\xi H$.

Let $(\rho,E_{\rho})$ be an irreducible smooth representation of $M$. Furthermore, we extend it to a representation of $P$ by letting $\rho|_U$ be trivial. Let $\varOmega$ be a $P$-stable locally closed subset of $G$. We set
\begin{equation*}
	I_c(\rho,\varOmega)=\left\{f:\varOmega\to E_{\rho}\ \middle|
	\begin{gathered}
		\ f \text{ is smooth, compactly supported mod} P,\\
		\text{and}\ f(px)=\delta_P(p)^{1/2}\rho(p)f(x)\\ \text{for all}\ p\in P\ \text{and}\ x\in \varOmega
	\end{gathered}
	\right\}.
\end{equation*}

Since $\mathcal{O}$ is open in $G$, we have the natural inclusion
\begin{equation*}
	\iota_{\rho}:I_c(\rho,\mathcal{O}) \hookrightarrow {\rm Ind}_P^G(\rho),
\end{equation*}
which is defined by extending functions to be zero outside $\mathcal{O}$. Since for each $\xi\in\varXi_0$, $P\xi H$ is open and closed in $\mathcal{O}$, we have the direct sum decomposition
\begin{equation*}
	I_c(\rho,\mathcal{O})=\bigoplus_{\xi\in\varXi_0}I_c(\rho,P\xi H).
\end{equation*}

By taking the $H$-invariants of the duals, we have
\begin{equation*}
	\Hom_H({\rm Ind}_P^G(\rho),\BC) \stackrel{\iota_{\rho}^{\ast}}{\longrightarrow} \Hom_H(I_c(\rho,\mathcal{O}),\BC) \simeq \bigoplus_{\xi\in\varXi_0}I_c(\rho,P\xi H).
\end{equation*}
Furthermore, we have the isomorphism
\begin{equation}\label{eq:iso}
	\Hom_{M\cap H}(\rho,\BC)\stackrel{\sim}{\longrightarrow} \Hom_H(I_c(\rho,PH),\BC),\ \lambda_0\mapsto \Lambda_0,
\end{equation}
given by
\begin{equation*}
	\langle \Lambda_0,f \rangle =\int_{M\cap H\backslash H} \langle \lambda_0,f(\dot{h}) \rangle d\dot{h}
\end{equation*}
for all $f\in I_c(\rho,PH)$.

\section{Representations of $\GL_2(F)$}\label{sec:rep_GL_2}
In this section, we will recall the representation theory of $\GL_2(F)$. In this section, let $G=\GL_2(F)$ and $P=MU$ be the standard Borel subgroup of $G$. First, we will recall how the irreducible admissible representations of $G$ are classified. The irreducible admissible representations of $G$ are as follow:
    \begin{itemize}[leftmargin=*]
        \item Supercuspidal representations.

        \item The quotient representations of ${\rm Ind}_P^G(\chi |\cdot|^{-1/2}\otimes\chi |\cdot|^{1/2})$  with respect to the $G$-stable subspace $W\coloneqq \langle (\chi|\cdot|^{1/2})\circ {\rm det} \rangle_{\BC}$ for all quasi-characters $\chi$ of $F^{\times}$. We call this the Steinberg representation twisted by $\chi$ and it is denoted by ${\rm St}_{\chi}$.

    	\item The principal series ${\rm Ind}_P^G(\chi_1\otimes\chi_2)$ for all quasi-characters $\chi_1$ and $\chi_2$ of $F^{\times}$ such that $\chi_1\chi_2^{-1}\neq |\cdot|_F^{\pm 1}$.

        \item Quasi-characters of $G$.
    \end{itemize}

Next, we will recall that the Jacquet module of ${\rm Ind}_P^G(\chi_1\otimes\chi_2)$ can be described explicitly. The Jacquet module of ${\rm Ind}_P^G(\chi_1\otimes\chi_2)$ is equivalent to the following two-dimensional representation:
    \begin{equation*}
		t=
		\begin{pmatrix}
			t_1 &  \\  & t_2
		\end{pmatrix}\mapsto
		\begin{cases}
			\begin{pmatrix}
				(\chi_1\otimes \chi_2)(t) &  \\  & (\chi_2\otimes\chi_1)(t)
			\end{pmatrix} & \text{if $\chi_1\neq \chi_2$,} \\\\
			\chi_1(t)
			\begin{pmatrix}
				1 & v(t_1/t_2) \\  & 1
			\end{pmatrix}& \text{if $\chi_1= \chi_2$,}
		\end{cases}
	\end{equation*}
where $v: F^{\times}\to \BZ$ is valuation map. Note that by the above theorem and the exactness of the Jacquet functor, if $\chi_1\chi_2^{-1}= |\cdot|_F^{-1}$, then the Jacquet module of ${\rm Ind}_P^G(\chi_1\otimes\chi_2)$ is equivalent to $\chi_2\otimes\chi_1$.

Let $H$ be a maximal $F$-split torus of $G$. Finally, we will recall the following criterion for an irreducible admissible representation of $G$ to be $H$-distinguished. Let $(\pi,V)$ be an infinite-dimensional irreducible admissible representation of $G$. Then $(\pi,V)$ is $H$-distinguished if and only if it admits a trivial central character. Furthermore, if the condition is satisfied, then $\dim\Hom_H(\pi,{\bf 1})=1$. (See for example \cite[Theorem 5.1.\ p.125]{MR3930015}.)

\section{proof for the $F$-split torus}\label{sec:main_result_in_the_split}
In this section, let ${\bf G}=\GL_2$, and let ${\bf P}={\bf M}{\bf U}$ be a standard Borel subgroup of ${\bf G}$. Namely, we set
\begin{equation*}
	{\bf P}= \left\{
	\begin{pmatrix}
		a & b \\  & c
	\end{pmatrix}
	\bigg |\ ac\neq 0\right\},\
	{\bf M}= \left\{
	\begin{pmatrix}
		a &  \\  & b
	\end{pmatrix}
	\bigg |\ ab\neq 0\right\} \text{and}\
	{\bf U}= \left\{
	\begin{pmatrix}
		1 & u \\  & 1
	\end{pmatrix}
	\right\}.
\end{equation*}
Let $\sigma={\rm Int}(w)$, where
\begin{equation*}
	w= \begin{pmatrix}
		0 & 1 \\ 1 & 0
	\end{pmatrix}.
\end{equation*}
Then $\bf H$ is
\begin{equation*}
	{\bf H}= \left\{
	\begin{pmatrix}
		a & b \\ b & a
	\end{pmatrix}
	\bigg |\ a^2-b^2\neq 0\right\}.
\end{equation*}
Classifying the $H$-distinguished representations is equivalent to classifying representations considered in Theorem \ref{thm:classification_in_the_split_case}. This is because for any irreducible admissible representation $\pi$ of $G$ and
\begin{equation*}
	q=
    \begin{pmatrix}
		1 & -1 \\ 1 & 1
	\end{pmatrix},
\end{equation*}
the map
\begin{equation*}
    \Hom_H(\pi,{\bf 1})\to \Hom_{qHq^{-1}}(\pi,{\bf 1}),\ \lambda\mapsto \lambda\circ\pi(q^{-1}),
\end{equation*}
for $\lambda\in\Hom_H(\pi,{\bf 1})$ is an isomorphism and $q{\bf H}q^{-1}={\bf M}$. Thus, we will consider $H$-distinguished representations.

We begin with some preparations. We define the maximal $\sigma$-split torus ${\bf S}_0$ by
\begin{equation*}
	\left\{
	\begin{pmatrix}
		a &  \\  & a^{-1}
	\end{pmatrix}
	\right\},
\end{equation*}
and we define the maximal $F$-split torus ${\bf A}\coloneqq {\bf M}$. Note the maximal $F$-split torus ${\bf A}$ contains ${\bf S}_0$. We take $\Delta=\{e_1-e_2 \}$ as the simple roots of $\Phi({\bf G},{\bf A})$, where
\begin{equation*}
	e_i\bigg(
	\begin{pmatrix}
		a_1 &  \\  & a_2
	\end{pmatrix}
	\bigg)
	=a_i\quad (i=1,2).
\end{equation*}
Thus, we have $\Delta^{\sigma}=\varnothing$, and hence ${\bf P}_{\Delta^{\sigma}}={\bf P}_{\varnothing}={\bf P}$. Therefore,
\begin{equation*}
	{\bf A}_{{\bf P}}={\bf A},\quad {\bf M}_{{\bf P}}={\bf A}.
\end{equation*}
Furthermore,
\begin{equation*}
	{\bf S}_{{\bf P}}={\bf S}_0.
\end{equation*}

We now describe the decomposition of $G$ into $(P,H)$-double cosets. By \cite[4\ p.722]{MR4121785}, the decomposition is given by
\begin{equation*}
	G=PH \sqcup PqH \sqcup PwqH\quad  (\text{disjoint}).
\end{equation*}
Since
\begin{equation*}
	(PqH)q^{-1}=P,\quad (PwqH)q^{-1}=
	\left\{
	\begin{pmatrix}
		\ast & \ast \\ \ast &
	\end{pmatrix}
	\right\},
\end{equation*}
both $PqH$ and $PwqH$ are closed but not open in $G$. Thus, among the $(P,H)$-double cosets in $G$, only $PH$ is open.

By the criterion in Section \ref{sec:rep_GL_2}, the irreducible $H$-distinguished representations are as follows:
\begin{itemize}[leftmargin=*]
	\item Irreducible $H$-distinguished supercuspidal representations.

    \item Steinberg representations twisted by a quadratic character $\mu$, which is square integrable.

	\item Principal series representations ${\rm Ind}_P^G(\chi\otimes\chi^{-1})$ for a quasi-character $\chi$ of $F^{\times}$, which is tempered if and only if $\chi$ is unitary.

    \item The trivial representation of $G$, which is not tempered.
\end{itemize}

\noindent Therefore, the proof of the main result reduces to establishing the following properties of the above representations:
\begin{enumerate}
	\item[Case 1] Neither the trivial representation of $G$ nor the principal series representations associated with a non unitary character $\chi$ are $H$-relatively tempered.

	\item[Case 2] The principal series representations associated with a unitary character $\chi$ are not $H$-relatively square integrable.

	\item[Case 3] The Steinberg representations twisted by nontrivial quadratic characters $\mu$ are not $H$-relatively cuspidal.

	\item[Case 4] The untwisted Steinberg representation is $H$-relatively cuspidal.
\end{enumerate}

\subsection{Proof of Case 1}
The trivial representation of $G$ is not $H$-relatively tempered since $G/Z_GH$ is not compact, where $Z_G$ is the center of $G$.

Consider the principal series representations ${\rm Ind}_P^G(\chi\otimes\chi^{-1})$ for a non unitary character $\chi$. 
We may write $\chi=\mu|\cdot|_F^\alpha$, with a quasi-character $\mu$ of  $F^{\times}$ and a positive number $\alpha$. Let $\rho=\mu|\cdot|_F^\alpha\otimes\mu^{-1}|\cdot|_F^{-\alpha}$, and let $(\pi,V)$ denote the induced representation ${\rm Ind}_P^G(\rho)$. Define the linear forms $L_1$ and $L_2$ (see \cite[Proposition\ 4.5.6\ p.478]{MR1431508}) on $V$ by
\begin{equation*}
	L_1(f)=f(1),\quad L_2(f)=\int_F f\left(
	\begin{pmatrix}
		& -1 \\ 1 &
	\end{pmatrix}
	\begin{pmatrix}
		1 & x \\  & 1
	\end{pmatrix}
	\right)dx
\end{equation*}
for $f\in V$. Since both $L_1$ and $L_2$ are $U$-invariant linear forms, they can be regarded as linear forms on the Jacquet module $(\pi_P,V_P)$. Define the functions $f_1$ and $f_2$ on $G$ by
\begin{align*}
	&f_1(g)=
	\begin{cases}
		(\delta_P^{1/2}\rho)(p)& \text{if $g=pk$ \text{with} $p\in P,k\in K_n$},\\
		0& \text{otherwise}.
	\end{cases}\\
	&f_2(g)=
	\begin{cases}
		(\delta_P^{1/2}\rho)(p)& \text{if $g=pw'u$ \text{with} $p\in P,u\in U_{\mathcal{O}_F}$},\\
		0& \text{otherwise},
	\end{cases}
\end{align*}
where
\begin{equation*}
	w'=\begin{pmatrix}
		& -1 \\ 1 &
	\end{pmatrix}
\end{equation*}
and
\begin{align*}
	&K_n\coloneqq \left\{
	\begin{pmatrix}
		a	& b \\ c & d
	\end{pmatrix}
	\in \GL_2(\mathcal{O}_F)\ \middle|\ a\equiv d\equiv 1,\ b\equiv c \equiv 0 \pmod{\varpi_F^n\mathcal{O}_F}
	\right\},\\
	&U_{\mathcal{O}_F}\coloneqq \left\{
	\begin{pmatrix}
		1	& x \\  & 1
	\end{pmatrix}
	\ \middle|\ x\in \mathcal{O}_F
	\right\}
\end{align*}
for some natural number $n$ such that $\mu$ is trivial on $1+\varpi_F^n\mathcal{O}_F$. Note that $f_1,f_2\in V$. Since $L_1(f_1)\neq 0, L_1(f_2)=0$, and $L_2(f_2)\neq 0$, we can take a dual basis $(v_1,v_2)$ of the basis $(L_1,L_2)$ such that $v_2=\bar{f_2}$, where $\bar{f}$ denotes the image of $f\in V$ in $V_P$. For the choice of the basis $(v_1,v_2)$, we have $\pi_P(m)v_1=\rho(m)v_1$ and $\pi_P(m)v_2=\rho^w(m)v_2$, where $\rho^w$ is defined by $\mu^{-1}|\cdot|_F^{-\alpha}\otimes\mu|\cdot|_F^{\alpha}$. Thus, for every $\lambda\in\Hom_H(\pi,{\bf 1})$, the $(H,\lambda)$-relative central exponent is either $\rho$ or $\rho^w$. Hence, if there exists $\lambda\in\Hom_H(\pi,{\bf 1})$ such that the $(H,\lambda)$-relative central exponent is $\rho^w$, then $\pi$ is not $H$-relatively tempered, since ${\rm Re}(\rho^w)\notin \overline{+\mathfrak{s}_{\bf P}^{\ast}}$.

We now proceed to prove the above claim. Assume that $\rho^w$ is not the $(H,\lambda)$-relative central exponent for any $\lambda\in\Hom_H(\pi,{\bf 1})$. Since $r_P(\lambda)$ is a linear form on $V_P$, it is of the form $aL_1+bL_2$ for some $a,b\in \BC$. However, we are assuming that $\rho^w$ is not $(H,\lambda)$-relative central exponent, namely $r_P(\lambda)$ vanishes on $\langle v_2 \rangle_{\BC}$, and hence $b=0$. Thus, we have
\begin{equation*}
    \langle r_P(\lambda),\pi_P(m)v \rangle =\rho(m)\langle r_P(\lambda),v \rangle
\end{equation*}
for all $v\in V_P$ and $m\in M$. By this equation and Proposition \ref{prop:r_P}, for any $f\in V$, there exists a positive number $\varepsilon\le 1$ such that
\begin{equation}\label{eq:contra}
	\langle \lambda, \pi(s)f \rangle = (\delta_P^{1/2}\rho)(s)\langle r_P(\lambda),\bar{f} \rangle
\end{equation}
for all $s\in S_{P}^{-}(\varepsilon)$. By setting $f=f_2$ in the above equation, the right-hand side is zero. Hence, if we can show the existence of $\lambda\in\Hom_H(\pi,{\bf 1})$ such that for every positive number $\varepsilon\le 1$, there exists $s\in S_{P}^{-}(\varepsilon)$ satisfying
\begin{equation*}
	\langle \lambda, \pi(s)f_2 \rangle \neq 0,
\end{equation*}
then equation \eqref{eq:contra} is a contradiction. Therefore, proving the claim reduces to showing the existence of such a $\lambda$. To this end, we prove the existence of such a $\lambda$.

We define the function $\Theta_{\alpha}$ on $G$ by
\begin{equation*}
	\Theta_\alpha(g)=\mu(d(\sigma(g)g^{-1}))|d(\sigma(g)g^{-1})|_F^{\alpha +1/2}
\end{equation*}
for $g\in G$, where $d$ denotes the map $G \to F$ such that for $g\in G$, $d(g)$ is the $(1,1)$-entry of $g$. The map $\Theta_{\alpha}$ is continuous on $G$ and satisfies
\begin{equation*}
	\Theta_\alpha(g)=
	\begin{cases}
		\delta_P^{-1/2}\rho^{-1}(p)& \text{if $g=ph$ \text{with} $p\in P,h\in H$},\\
		0& \text{otherwise}.
	\end{cases}
\end{equation*}
This shows that $\Theta_{\alpha}$ satisfies $\Theta_\alpha(pgh)=\delta_P^{-1/2}\rho^{-1}(p)\Theta_\alpha(g)$ for all $p\in P,h\in H$ and $g\in G$. Thus, if we define the linear form $\lambda$ on $V$ by
\begin{equation*}
	\langle \lambda, f \rangle =\int_{P\backslash G}\Theta_\alpha(g)f(g)dg
\end{equation*}
for all $f\in V$, then $\lambda$ is $H$-invariant. Let $s_n\in S_{P}^{-}(1)$ and $g\in Pw'U_{\mathcal{O}_F}$ be the elements defined by
\begin{equation*}
	s_n=
	\begin{pmatrix}
		\varpi^n & \\  & \varpi^{-n}
	\end{pmatrix},\quad
	g=
	\begin{pmatrix}
		a & b \\  & c
	\end{pmatrix}
	\begin{pmatrix}
		& -1 \\ 1 &
	\end{pmatrix}
	\begin{pmatrix}
		1 & u \\  & 1
	\end{pmatrix},
\end{equation*}
where $n\ge 1,\ a,c\in F^{\times}, b\in F$ and $u\in\mathcal{O}_F$. We have
\begin{equation}\label{eq:gs_n}
	gs_n=
	\begin{pmatrix}
		a\varpi^{-2n} & b\varpi^{2n} \\  & c\varpi^{2n}
	\end{pmatrix}
	\begin{pmatrix}
		& -1 \\ 1 &
	\end{pmatrix}
	\begin{pmatrix}
		1 & u\varpi^{-2n} \\  & 1
	\end{pmatrix}.
\end{equation}
By \eqref{eq:gs_n}, we see that the support of $\pi(s_n)f_2$ is equal to $Pw'U_{\varpi^{2n}\mathcal{O}_F}$, where $U_{\varpi^{2n}\mathcal{O}_F}$ denotes the set of matrices of the form
\begin{equation*}
	\begin{pmatrix}
		1& u\varpi^{2n} \\  & 1
	\end{pmatrix}
\end{equation*}
with $u\in\mathcal{O}_F$. Let $g\in Pw'U_{\varpi^{2n}\mathcal{O}_F}$ be the element defined by
\begin{equation*}
	g=
	\begin{pmatrix}
		a & b \\  & c
	\end{pmatrix}
	\begin{pmatrix}
		& -1 \\ 1 &
	\end{pmatrix}
	\begin{pmatrix}
		1 & u \\  & 1
	\end{pmatrix},
\end{equation*}
where $n\ge 1,\ a,c\in F^{\times}, b\in F,\ u\in\varpi^{2n}\mathcal{O}_F$, then we have
\begin{equation}\label{eq:g}
	g=
	\begin{pmatrix}
		a(u^2-1)^{-1} & -au(u^2-1)^{-1} \\  & c
	\end{pmatrix}
	\begin{pmatrix}
		u & 1 \\ 1 & u
	\end{pmatrix}.
\end{equation}
Therefore, by \eqref{eq:gs_n} and \eqref{eq:g}, if $n$ is sufficiently large, then for all $g\in Pw'U_{\varpi^{2n}\mathcal{O}_F}$, we see
\begin{equation*}
	\Theta_\alpha(g)(\pi(s_n)f_2)(g)=q_F^{4n\alpha}.
\end{equation*}
This shows that $\langle \lambda, \pi(s_n)f_2 \rangle \neq 0$. From the above argument, we see that $\rho^w$ is an $(H,\lambda)$-relative central exponent. Thus, we conclude that the principal series representations associated with a non unitary character $\chi$ are not $H$-relatively tempered.

\subsection{Proof of Case 2}
Since $\chi$ is unitary, it follows from the form of the Jacquet module of ${\rm Ind}_P^G(\chi\otimes\chi^{-1})$ that the relative exponent along $P$ associated with an $H$-invariant linear form is trivial. Thus ${\rm Ind}_P^G(\chi\otimes\chi^{-1})$ is not $H$-relatively square integrable.

\subsection{Proof of Case 3}
Let $\rho=\mu |\cdot|_F^{-1/2}\otimes\mu^{-1}|\cdot|_F^{1/2}$ and the representation space of ${\rm Ind}_P^G(\rho)$ be denoted by $V$. Let $W$ be $\langle \mu\circ {\rm det} \rangle_{\BC}$, which is an invariant subspace of $V$, so $V/W$ realizes the Steinberg representation. For the exact sequence of $G$-modules
\begin{equation*}
	0\to W\to V\to V/W\to 0,
\end{equation*}
by taking $H$-invariant of the dual, we have
\begin{equation*}
	0\to\Hom_H(V/W,\BC) \to \Hom_H(V,\BC)\to \Hom_H(W,\BC).
\end{equation*}
Since ${\rm det}|_H$ is surjective, we have $\Hom_H(W,\BC)=0$.\ Hence $\Hom_H(V/W,\BC) \to \Hom_H(V,\BC)$ is surjective. By the proof of Case 1, ${\rm Ind}_P^G(\rho)$ is not $H$-relatively cuspidal, and thus the Steinberg representations twisted by  nontrivial quadratic characters $\mu$ are not $H$-relatively cuspidal.

\subsection{Proof of Case 4}
Let $\rho=|\cdot|_F^{-1/2}\otimes |\cdot|_F^{1/2}$ and let the representation space of ${\rm Ind}_P^G(\rho)$ be denoted by $V$. Let $W$ be the set of all constant functions on $G$, which is an invariant subspace of $V$. Then $V/W$ realizes the Steinberg representation. Let $(\pi,V/W)$ denote the Steinberg representation.

Since $\delta_P^{1/2}\rho=1$, we have
\begin{align*}
	{\rm Ind}_P^G(\rho)&=C_c^{\infty}(P\backslash G),\\ I_c(\rho,PH)&=C_c^{\infty}(P\backslash PH),\\ I_c(\rho,G\smallsetminus PH)&=C_c^{\infty}((P\backslash G)\smallsetminus(P\backslash PH))=C_c^{\infty}(P\backslash (PqH\cup PwqH)),
\end{align*}
where $C_c^{\infty}(X)$ denotes the space of all smooth functions on $X$ with compact support. By \cite[Proposition\ 4.3.1\ p.436]{MR1431508}, we have the exact sequence
\begin{equation*}
    0\to I_c(\rho,PH) \to {\rm Ind}_P^G(\rho)\to I_c(\rho,G\smallsetminus PH) \to 0.
\end{equation*}
Note that $I_c(\rho,G\smallsetminus PH)$ decomposes as follows:
\begin{equation*}
	I_c(\rho,G\smallsetminus PH)=I_c(\rho,PqH \sqcup PwqH)=I_c(\rho,PqH)\sqcup I_c(\rho,PwqH).
\end{equation*}
Let $\xi=q$ or $wq$. Then we have
\begin{equation*}
	f(p\xi h)=f((p\xi h \xi^{-1}) \xi)=(\delta_P^{1/2}\rho)(p\xi h \xi^{-1})f(\xi) =f(\xi)
\end{equation*}
for all $p\in P,\ h\in H$ and $f\in I_c(\rho,P\xi H)$. Hence, if we define a linear form $\lambda$ on $V$ by
\begin{equation*}
	\lambda(f)=f(q)-f(wq)
\end{equation*}
for all $f\in V$, then $\lambda$ is nontrivial and $H$-invariant. Furthermore, $\lambda$ vanishes on the subspace $W$ of all constant functions. Thus, $\lambda$ can be regarded as a linear form on $V/W$. Note that the dimension of $\Hom_H(V/W,\BC)$ is at most $1$. Thus, if we show that $(\pi,V/W)$ is $(H,\lambda)$-relatively cuspidal, then we conclude that it is $H$-relatively cuspidal.

We now proceed to prove the above claim. We will show that $(\pi,V/W)$ is $(H,\lambda)$-relatively cuspidal by Theorem \ref{thm:criterion_of_relatively_cusp}. If ${\bf MH}(F)=MH$, then it follows from Lemma \ref{lem:conjugate_of_sigma-split} that all $\sigma$-split parabolic subgroups are $H$-conjugate to $P$. We see that the above assumption holds as follows. Let $A\in{\bf M}$ and $B\in{\bf H}$ be given by
\begin{equation*}
	A=
	\begin{pmatrix}
		a & \\  & b
	\end{pmatrix}
	,\ B=
	\begin{pmatrix}
		c & d \\ d & c
	\end{pmatrix}.
\end{equation*}
If
\begin{equation*}
		AB=
	\begin{pmatrix}
		ac & ad \\ bd & bc
	\end{pmatrix}
\end{equation*}
is an $F$-rational point, then $ab^{-1}\in F$. Thus, the matrix
\begin{equation*}
	\begin{pmatrix}
		ab^{-1} &  \\  & 1
	\end{pmatrix}
	AB=
	\begin{pmatrix}
		bc & bd \\ bd & bc
	\end{pmatrix}
\end{equation*}
is also an $F$-rational point. Therefore, we have
\begin{equation*}
	AB=
	\begin{pmatrix}
		ab^{-1} &  \\  & 1
	\end{pmatrix}
	\begin{pmatrix}
		bc & bd \\ bd & bc
	\end{pmatrix}
	\in MH.
\end{equation*}
Thus, by the second assertion of Theorem \ref{thm:criterion_of_relatively_cusp}, it is sufficient to show that $r_P=0$. We will calculate $r_P(\lambda)$ by Proposition \ref{prop:r_P}. Let $f\in V$. For $n\in \BN$, we define $s_n\in S_P$ by
\begin{equation*}
	s_n=\begin{pmatrix}
		\varpi_F^{n} &  \\  & \varpi_F^{-n}
	\end{pmatrix}.
\end{equation*}
We have
\begin{align*}
	qs_n &=
	\begin{pmatrix}
		1 &-1 \\ 0 & 1
	\end{pmatrix}
	\begin{pmatrix}
		2\varpi_F^n &0 \\ 0 & -\varpi_F^{-n}
	\end{pmatrix}
	\begin{pmatrix}
		1 & 0 \\ -\varpi_F^{2n} & 1
	\end{pmatrix},\\
	qws_n &=
	\begin{pmatrix}
		1 &-1 \\ 0 & 1
	\end{pmatrix}
	\begin{pmatrix}
		2\varpi_F^n &0 \\ 0 & -\varpi_F^{-n}
	\end{pmatrix}
	\begin{pmatrix}
		1 & 0 \\ \varpi_F^{2n} & 1
	\end{pmatrix}.
\end{align*}
Thus, if $n$ is sufficiently large, then we have
\begin{equation*}
	f(qs_n)=f(1)=f(qws_n).
\end{equation*}
This shows that $\langle \lambda,\pi(s_n)f \rangle=0$. Note that for every number $\varepsilon >0$, there is a sufficiently large natural number $n$ such that $s_n\in S_P^{-}(\varepsilon)$. By Proposition \ref{prop:r_P}, for such $n$, we have
\begin{equation*}
	\langle \lambda,\pi(s_n)f \rangle =\delta_P^{1/2}(s_n)\langle r_P(\lambda),\pi_P(s_n)\bar{f} \rangle.
\end{equation*}
Since the left side of the above equation vanishes, we have
\begin{equation*}
	\langle r_P(\lambda),\pi_P(s_n)\bar{f} \rangle =0.
\end{equation*}
Since the Jacquet module of the Steinberg representation is one-dimensional, this shows that
\begin{equation*}
	\langle r_P(\lambda),j_P(f) \rangle =0.
\end{equation*}
Namely $r_P=0$. Thus, we conclude that the untwisted Steinberg representation is $H$-relatively cuspidal.

\section{proof for the non $F$-split case}\label{sec:main_result_in_the_non-split}
When $H$ is non $F$-split, the proof is almost the same. In this section, let ${\bf G}=\GL_2$, and let ${\bf P}={\bf M}{\bf U}$ be a standard Borel subgroup of ${\bf G}$. Namely, we set
\begin{equation*}
	{\bf P}= \left\{
	\begin{pmatrix}
		a & b \\  & c
	\end{pmatrix}
	\bigg |\ ac\neq 0\right\},\
	{\bf M}= \left\{
	\begin{pmatrix}
		a &  \\  & b
	\end{pmatrix}
	\bigg |\ ab\neq 0\right\} \text{and}\
	{\bf U}= \left\{
	\begin{pmatrix}
		1 & u \\  & 1
	\end{pmatrix}
	\right\}.
\end{equation*}
Let $\sigma={\rm Int}(w)$, where
\begin{equation*}
	w= \begin{pmatrix}
		0 & 1 \\ \tau & 0
	   \end{pmatrix}
       \quad \tau\in F^{\times}\smallsetminus (F^{\times})^2.
\end{equation*}
Then $\bf H$ is
\begin{equation*}
	{\bf H}= \left\{
	\begin{pmatrix}
		a & b \\ \tau b & a
	\end{pmatrix}
	\bigg |\ a^2-\tau b^2\neq 0\right\}.
\end{equation*}
In this setting, similarly to Section \ref{sec:main_result_in_the_split}, we can take the $\sigma$-split torus ${\bf S}_0$, the maximal $F$-split torus ${\bf A}\coloneqq {\bf M}$ containing it, and the simple roots $\Delta=\{e_1-e_2 \}$ of $\Phi({\bf G},{\bf A})$.

We now describe the decomposition of $G$ into $(P,H)$-double cosets. The decomposition is given by
\begin{equation*}
	G=PH.
\end{equation*}
In fact, for
\begin{equation*}
	\begin{pmatrix}
		a & b \\ c & d
	\end{pmatrix}
    \in G,
\end{equation*}
we can find
\begin{equation*}
	\begin{pmatrix}
		a' & b' \\  & 1
	\end{pmatrix}
    \in P\ \text{and}\
    \begin{pmatrix}
    	x & y \\ \tau y & x
    \end{pmatrix}
    \in H,
\end{equation*}
where
\begin{equation*}
x=d,\quad y=\tau^{-1}c,\quad
	\begin{pmatrix}
		d & c \\ \tau^{-1}c & d
	\end{pmatrix}
    \begin{pmatrix}
    	a' \\ b'
    \end{pmatrix}
    =
    \begin{pmatrix}
       	a \\ b
    \end{pmatrix}.
\end{equation*}

Since $H$ contains the center of $G$ and in view of considerations in Section \ref{sec:main_result_in_the_split}, the proof of the main result reduces to establishing the following properties:
\begin{enumerate}
	\item[Case 1] Neither the trivial representation of $G$ nor the principal series representations associated with a non unitary character $\chi$ are $H$-relatively tempered.

	\item[Case 2] The principal series representations associated with a unitary character $\chi$ are not $H$-relatively square integrable.

	\item[Case 3] The Steinberg representations twisted by  quadratic characters $\mu$ are not $H$-relatively cuspidal.
\end{enumerate}

The proofs of Case 1 and Case 2 are shown in the same way as in Section \ref{sec:main_result_in_the_split}.

\subsection{Proof of Case 3}
First, we consider whether a Steinberg representation with trivial central character is $H$-distinguished. For a quadratic character $\mu$ of $F^{\times}$, let $\rho=\mu |\cdot|_F^{-1/2}\otimes\mu^{-1}|\cdot|_F^{1/2}$. Let the representation space of ${\rm Ind}_P^G(\rho)$ be denoted by $V$. Let $W$ be $\langle \mu\circ {\rm det} \rangle_{\BC}$, which is an invariant subspace of $V$. Then $V/W$ realizes the Steinberg representation. We consider the exact sequence of $H$-periods
\begin{equation}\label{eq:exact_seq_H-duals}
	0\to\Hom_H(V/W,\BC) \to \Hom_H(V,\BC)\to \Hom_H(W,\BC).
\end{equation}
Note that $\dim \Hom_H(V,\BC)=1$  by \eqref{eq:iso}. Suppose that $\mu$ corresponds to $F(\sqrt{\tau})$ by local class field theory, or is trivial. Since $\lambda$ is defined in Case 1 of Section \ref{sec:main_result_in_the_split}, the third map of the above exact sequence is nontrivial. Hence, we have $\Hom_H(V/W,\BC)=0$. Namely, in this case, the Steinberg representation is not $H$-distinguished. While, suppose that $\mu$ does not correspond to $F(\sqrt{\tau})$ by local class field theory, and is not trivial. By $\Hom_H(W,\BC)=0$ and the above exact sequence, the Steinberg representation is $H$-distinguished.

Finally, we will show that the $H$-distinguished Steinberg representations are not $H$-relatively cuspidal. By $G=PH$ and the result in Section \ref{sec:induction_from_split_parabolic_subgp}, the second map of the exact sequence \eqref{eq:exact_seq_H-duals} is surjective. Moreover, by the proof in Case 1 of Section \ref{sec:main_result_in_the_split}, ${\rm Ind}_P^G(\rho)$ is not $H$-relatively cuspidal. Thus, the Steinberg representations are not $H$-relatively cuspidal.

\bibliographystyle{amsalpha}
\bibliography{class_dist_rep_GL(2)_ref}
\nocite{*}
\end{document}